\documentclass[12pt]{amsart}

\usepackage{amsmath, amssymb}
\usepackage{tikz} 
\usepackage{accents}

\DeclareMathOperator{\Sp}{Sp}

\newcommand{\cO}{{\mathcal O}}
\newcommand{\cstar}{$\mathrm{C}^*$}

\newcommand{\bbT}{{\mathbb T}}

\newcommand{\bbN}{{\mathbb N}}
\newcommand{\bbC}{\mathbb C}

\DeclareMathOperator{\id}{id}

\newcommand{\cK}{\mathcal K}

\newtheorem{thm}{Theorem}
\newtheorem{theorem}[thm]{Theorem}

\newtheorem{lemma}[thm]{Lemma}

\theoremstyle{definition}

\newcounter{my_enumerate_counter}
\newcommand{\pushcounter}{\setcounter{my_enumerate_counter}{\value{enumi}}}
\newcommand{\popcounter}{\setcounter{enumi}{\value{my_enumerate_counter}}}

\newcommand{\bfs}{\mathbf s}

\DeclareMathOperator{\Ad}{Ad}

\newcommand{\bbF}{\mathbb F}

\DeclareMathOperator{\Ext}{Ext}

\def\utilde#1{{\baselineskip=0pt\vtop{\hbox{$#1$}\hbox{$\sim$}}}{}}

\title{The Calkin algebra is not countably homogeneous}
\author{Ilijas Farah} 

\address{Department of Mathematics and Statistics\\
York University\\
4700 Keele Street\\
North York, Ontario\\ Canada, M3J 1P3}
\email{ifarah@yorku.ca}
\urladdr{http://www.math.yorku.ca/$\sim$ifarah}

\author{Ilan Hirshberg} 

\address{Department of Mathematics\\
 Ben Gurion University of the Negev\\
  P.O.B. 653, Be'er\\
Sheva 84105, Israel}
\email{ilan@math.bgu.ac.il}

\urladdr{http://www.math.bgu.ac.il/~ilan/}
\date{\today}

\date{\today}
\begin{document}
\begin{abstract}
We show that the Calkin algebra is not countably homogeneous, in the sense of continuous model theory. We furthermore show that the connected component of the unitary group of the Calkin algebra is not countably homogeneous.
\end{abstract}
\maketitle

Motivated by their study of extensions of \cstar-algebras, Brown, Douglas and Fillmore asked 
whether the Calkin algebra has a $K$-theory reversing  automorphism and whether it has outer automorphisms at all
   (\cite[Remark~1.6 (ii)]{BrDoFi}). By \cite{PhWe:Calkin} and \cite{Fa:All} the answer to the latter question is independent from ZFC. In particular, since inner automorphisms fix $K$-theory, 
    a negative answer to the former question is relatively consistent with ZFC. 
   It is not known whether the existence of a $K$-theory reversing automorphism of the Calkin algebra is relatively consistent with ZFC. 
   All known automorphisms of the Calkin algebra (\cite{PhWe:Calkin} and \cite[\S 1]{Fa:All}) 
   act trivially on its $K$-theory, as they are implemented by a unitary on every separable subalgebra of the Calkin algebra.
   
A scenario for using Continuum Hypothesis 
to construct a $K$-theory reversing automorphism of 
the Calkin algebra on separable Hilbert space, denoted~$Q$, was sketched in  \cite[\S 6.3]{FaHaSh:Model2} and in \cite[\S 7.1]{Fa:Logic}. 
The following theorems demonstrate that this strategy 
does not work and suggest that  question of the existence of such automorphism 
is even more difficult than previously thought  (for terminology see below and  \cite{FaHaSh:Model2} or~\cite{FaHa:Countable}). 

\begin{theorem} \label{T1} The Calkin algebra $Q$ is not countably homogeneous, 
and this is witnessed by a quantifier-free type. 
\end{theorem} 

\begin{theorem} \label{T2} The group $U_0(Q)$ of Fredholm index zero unitaries in $Q$ is not countably 
homogeneous. 
\end{theorem}

Our theorems give negative answers to \cite[Questions~5.2 and 5.7]{FaHa:Countable}
and a novel obstruction to countable saturation of $Q$. 
In   \cite[Question~5.1]{FaHa:Countable} it was asked whether all obstructions to (quantifier-free) countable  saturation of $Q$ are 
of $K$-theoretic nature. The obstruction given in our results is finer than the Fredholm index, but it is 
 $K$-homological and therefore ultimately $K$-theoretical.
 In addition, the obstruction given in Theorem~\ref{T1} is quantifier-free and   one
given in  Theorem~\ref{T2} appears to have little to do with the Fredholm index.   
It should be noted that one of the key ideas, using $\Ext(M_{2^\infty})$, is due to N.C. Phillips, and it was already 
used in the proof of  \cite[Proposition~4.2]{FaHa:Countable}.

Model theory of \cstar-algebras and their unitary groups is based on  \cite{BYBHU} 
and described in \cite[\S 2.3.1 and \S 2.3.3, respectively]{FaHaSh:Model2}. 
\emph{Formulas} in logic of metric structures are defined recursively.   
 In case of \cstar-algebras, atomic formulas are expressions of the form $\|t(\bar x)\|$
 where $t$ is a noncommutative *-polynomial in variables $\bar x=(x_1, \dots, x_n)$. 
 The set of all formulas is the smallest set $\bbF$ containing all atomic formulas 
such  that (i) for every~$n$, all continuous $f\colon [0,\infty)^n\to [0,\infty)$ and all $\phi_1, \dots, \phi_n$ in $\bbF$
the expression $f(\phi_1, \dots, \phi_n)$ belongs to  $\bbF$ and 
(ii) if $\phi\in \bbF$,  $m\geq 1$, and $x$ is a variable symbol  than both  $\sup_{\|x\|\leq m} \phi$ and $\inf_{\|x\|\leq m} \phi$  belong to $\bbF$
 (see  \cite[\S 2.4]{FaHaSh:Model2}). 
If $\phi(x_1, \dots, x_n)$ is a formula, $A$ is a \cstar-algebra, and $a_1, \dots, a_n$ are elements of 
$A$, then the \emph{interpretation} $\phi(a_1, \dots, a_n)^A$ is obviously defined by recursion. 
 A \emph{condition} is any expression of the form $\phi\leq r$  for formula $\phi$ and $r\geq 0$
 and \emph{type} is a set of conditions
 (\cite[\S 4.3]{FaHaSh:Model2}).   As every expression of the form $\phi=r$ is equivalent to the condition $\max(\phi, r)\leq r$
 and every expression of the form $\phi\geq r$ is equivalent  to the condition $\min(0,r-\phi)\leq 0$, we shall freely refer to such expressions as conditions. 
For $n\geq 1$, the \emph{$n$-type} is a type such that free variables occurring in its conditions are included in $\{x_1, \dots, x_n\}$. 
It is important that  each free variable $x$  is associated with a domain of quantification, which in our case
reduces to asserting that $\|x\|\leq m$ for some fixed $m$.
  
Given a \cstar-algebra $A$ and sequence $\bar a=(a_j: n\in \bbN)$ in $A$, the \emph{type of $\bar a$ in $A$} is
the set of all conditions 
$\phi(x_1, \dots, x_m)\leq r$ 
such that $\phi(a_1, \dots, a_m)^A\leq r$. 
A structure $C$ is said to be \emph{countably homogeneous} if for every two sequences
 $\bar a=(a_n:n\in \bbN)$ and $\bar b=(b_n:n\in \bbN)$ with the same type
 and every $c\in C$ there exists $d\in C$ such that  $(\bar a, c)$ and $(\bar b, d)$ have the same type. 
 Our proof of the failure of countable homogeneity in $Q$  will show that 
sequences $\bar a$ and $\bar b$ can be chosen to be finite. 

We recall the definitions the semigroups $\Ext(A)$ and $\Ext^w(A)$. 
 If $A$ is a unital \cstar-algebra, we consider injective unital *-homomorphisms $\pi \colon A \to Q$ (such a *-homomorphism is the Busby invariant of an extension of $A$ by the compacts). By slight abuse of notation, we call such a *-homomorphism an extension. Two extensions $\pi_j\colon A\to Q$, for $j=1,2$ are said to be weakly equivalent if there is a unitary $u$ in the Calkin algebra
such that $\pi_1=\Ad u\circ \pi_2$. If $u$ above is furthermore required to have Fredholm index zero then we say that these extensions  are equivalent. The set of such *-homomorphisms is equipped with the direct sum operation (using implicitly the fact that $M_2(Q) \cong Q$), and the set of equivalence relations forms a semigroup, denoted  $\Ext^w(A)$ or $\Ext(A)$, respectively.
They correspond to semigroups 
$\Ext^u_*(A,\cK)$ where $*=s,w$ and $\cK$ denotes the algebra of compact operators on separable Hilbert space
as defined in \cite[Definition~15.6.3, Proposition~15.6.2 and \S 15.4 (2), (3)]{blackadar1998k}    
(see also    \cite{BrDoFi,higson2000analytic,Arv:Notes}).

If $A$ is a simple unital \cstar-subalgebra of $Q$ and $p\in A'\cap Q$  is a 
nonzero projection, then   $a\mapsto pap$ is an injective unital *-homomorphism from $A$ into $pQp\cong Q$. The isomorphism between $pQp$ and $Q$ used here is chosen by picking an isometry $v$ such that $vv^* = p$, and the map $Q \to pQp$ is given by $x \mapsto vxv^*$. The choice of $v$ is unique up to multiplication by a unitary, and therefore it does not affect the $\Ext^w$ class. 
(The choice of $v$ can affect the $\Ext$ class, and therefore the choice of $p$ only determines the weak equivalence class.)    
Therefore projections in $A' \cap Q$ 
determine $\Ext^w$-classses of 
 unital extensions of $A$, after identifying $pQp$ with $Q$ in the manner we described.

A subalgebra $A$ of $Q$ is \emph{split} if 
there is a unital *-homomorphism from $\Phi\colon A\to B(H)$ such that $\pi\circ\Phi=\id_A$. 
The following lemma   is related 
to   \cite[Lemma~5.1.2 and Lemma~5.1.2]{higson2000analytic}.

\begin{lemma} \label{L1} Let $A$ be a simple separable unital subalgebra of $Q$ and let $p$ and $q$ be projections in $A'\cap Q$. 
Then $p$ and $q$ are Murray-von Neumann equivalent in $A'\cap Q$ if and only if the extensions of $A$ corresponding to 
$p$ and $q$ are weakly equivalent. 
\end{lemma}

\begin{proof} The direct implication is trivially true because 
 of our convention that for nonzero $p\in A'\cap Q$ we identify $pQp$ with $Q$ 
 and $p$  with unital extension $a\mapsto pap$ 
 of $A$.  
 We now prove the converse implication. If the extensions corresponding to $p$ and $q$ are weakly equivalent, then 
there exists a partial isometry $v$ in $Q$ such that $v^*v=p$, $vv^*=q$, 
and $vpapv^*=qaq$ for all $a\in A$. It will suffice to check that $v\in A'\cap Q$. 
Fix $a\in A$. We have $vav^*=vv^*avv^*$, and since $vv^*\in A'\cap Q$, 
we have $vv^*avv^*=avv^*$ and therefore $vav^*=avv^*$. Multiplying by $v$ on the right hand side and using $v^*v\in A'\cap Q$
we have $vv^*va=avv^*v$. But since $v$ is a partial isometry we have $vv^*v=v$, thus showing that $va=av$. 
\end{proof} 

If $A$ is a  separable, unital and nuclear \cstar-algebra then $\Ext^w(A)$ is a group  (\cite[p. 586]{choi1976completely}). This implies that every extension of $A$ corresponds to some $p\in A'\cap Q$. To see that, if $\pi_1 \colon A \to Q$ is any given extension, then there exists an extension $\pi_2 \colon A \to Q$ such that $\pi_1 \oplus \pi_2$ is weakly equivalent to $\id_A$. The extension $\pi_1$ corresponds to the projection $\left ( \begin{matrix} 1 & 0 \\ 0 & 0 \end{matrix} \right ) \in (\pi_1 \oplus \pi_2)(A)' \cap M_2(Q)$. When we identify $\pi_1 \oplus \pi_2$ with $\id_A$ via a unitary and an isomorphism $M_2 (Q) \cong Q$, the above matrix is identified with a projection $p \in A' \cap Q$ as required.


\begin{lemma}\label{L.O2} Let $A$ be a separable unital  subalgebra of $Q$ such that $\Ext^w(A)$ is a group. 
Then the Cuntz algebra $\cO_2$ unitally embeds into $A'\cap Q$ if and only if
$A$ is split. 
\end{lemma}  

\begin{proof} Assume first that $A$ is split.
Recall that by Voiculescu's theorem (\cite{voiculescu1976non}, \cite[Section 4]{Arv:Notes}), all trivial extensions of $A$ are equivalent. In particular, $\id_A$ is equivalent to $\id_A \oplus \id_A \colon A \to Q \otimes M_2$.
Thus,  
if $A$ is split then there is a  projection $p\in A'\cap Q$ such that both $pAp$ and $(1-p)A(1-p)$ are split in $pQp$ and $(1-p)Q(1-p)$, 
respectively. Lemma~\ref{L1} implies that 1, $p$ and $1-p$ are Murray--von Neumann equivalent and therefore 
 $\cO_2$ embeds unitally into $A'\cap Q$.

Now assume  $\cO_2$ unitally embeds into $A'\cap Q$. Then the extension of $A$ corresponding to $1$ is an idempotent in $\Ext(A)$. 
Since the identity is the only idempotent in a group,~$A$ is split. 
\end{proof}

Let $A$ denote the CAR algebra, $M_{2^\infty}$. $A$ is singly generated by~\cite{OlsZam}. Fix a generator $g$ for $A$. 
Since $A$ is nuclear, $\Ext^w(A)$ is a group.

\begin{lemma}\label{L.CAR} 
For every unital extension $\pi$ of $A$, the type of  $\pi(g)$ in $Q$ is the same as the  type of $\pi_0(g)$  
corresponding to the trivial extension $\pi_0$ of $A$.
\end{lemma}

\begin{proof} Represent $A$ as a direct limit of $M_{2^n}(\bbC)$ and choose $a_n\in M_{2^n}(\bbC)$ 
such that $\lim_n a_n=g$. 
Fix a unital extension $\pi$ of $A$. For  $n\in \bbN$
 the group $\Ext^w(M_{2^n}(\bbC))$ is trivial, and therefore the type of
$\pi(a_n)$ in $Q$ is  the same as the type of $\pi_0(a_n)$ in $Q$. 
Fix a formula $\phi(x)$. Since the interpretation $b\mapsto \phi(b)^Q$ is continuous, we 
have 
\[
\phi(\pi(g))^Q=\lim_n \phi(\pi(a_n))^Q=\lim_n \phi(\pi_0(a_n))^Q=\phi(\pi_0(g))^Q. 
\]
Since $\phi$ was arbitrary, the conclusion follows. 
\end{proof}

\begin{proof}[Proof of Theorem~\ref{T1}] 
Fix a unital *-homomorphism $\pi$ of $A$ into $Q$ and consider the 2-type in $x_1, x_2$  consisting of conditions
\[
x_j\pi(g)=\pi(g)x_j, \qquad x_j^*x_j=1, \qquad x_1x_1^*+x_2x_2^*=1
\]
for $j=1,2$. 
By Lemma~\ref{L.O2}, this type is realized if and only if $\pi$ is the trivial extension.  

Since $A$ has both trivial and nontrivial extensions (as a matter of 
fact, $\Ext^w(A)$ is uncountable by \cite[Proposition~3]{Cho:Strong} or  \cite{PiPo:On}) and the type of $\pi(g)$ 
does not depend on the choice of the extension $\pi$ by Lemma~\ref{L.CAR}, 
$Q$ is not (countably) homogeneous. 
\end{proof}

The salient point in our proof of Theorem~\ref{T2} is the fact that the presence of $\cO_2$ in $A'\cap Q$ can be detected 
from $A'\cap U_0(Q)$. We note that in   \cite[Theorem~4.6]{paterson1983harmonic} it was shown that 
 if $B$ and $C$  are simple \cstar-algebras such that their unitary groups are isometrically isomorphic
 then this isomorphism extends to an isomorphism or an anti-isomorphism of $B$ and $C$. 
 We were not able, however, to use this result directly.

By Voiculescu's theorem (\cite{voiculescu1976non}) for a unital separable \cstar-subalgebra $A$ of $Q$ 
one has $(A'\cap Q)'=A$ and $Z(A'\cap Q)=Z(A)$. 
We need the following self-strengthening of this result. 
 
\begin{lemma} \label{L.center} 
If $A$ is a unital separable \cstar-subalgebra of $Q$ then $(A'\cap U_0(Q))'=A$
and $Z(A'\cap U_0(Q))=Z(A)\cap U_0(Q)$.
\end{lemma} 

\begin{proof} 
Assume $b\in Q$ is such that $b\notin A$. Since $A = (A' \cap Q)'$, there exists an element $x \in A' \cap Q$ such that $xb \neq bx$. By replacing $x$ by its real or imaginary part, we may assume that $x$ is self-adjoint, and we may assume that $\|x\|<\pi$. Set $u=\exp(ix)$. Then $u \in A' \cap U_0(Q)$ and since $x \in \mathrm{C}^*(u)$, we have $ub \neq bu$.
Therefore $b\notin (A'\cap U_0(Q))'$. 
Since $A= (A'\cap Q)'$ and $b$ was arbitrary, this proves $(A'\cap U_0(Q))'=A$. 

The second equality is a standard consequence of the first. 
If $b\in Z(A'\cap U_0(Q))$, then by the above
$b\in A$ and therefore $b\in Z(A)$. Since $Z(A)\subseteq A'\cap Q$, the conclusion follows. 
\end{proof}

\begin{lemma} \label{L.O2.x} 
The Cuntz algebra $\cO_2$ is the universal \cstar-algebra generated by three unitaries $u$, $v$ and $w$ satisfying the following relations: 
 \begin{enumerate}
 \item \label{L.O2.1} $u^2=v^3=w^6=1$.
\item $\|w-1\|=1$.
\item $uw^3u=-w^3$. 
\item \label{L.O2.4} $vw^2v^2=e^{2\pi i/3} w^2$.
\pushcounter
 \end{enumerate}
\end{lemma} 

\begin{proof} 
In \cite[Theorem~2.6]{choi1979simple} Choi proved that every \cstar-algebra  generated by 
  unitaries $u$ and $v$ and projection $p$ satisfying the following conditions is isomorphic to $\cO_2$:
 \begin{enumerate}
 \popcounter
 \item $u^2=v^3=1$. 
 \item $p+upu=1$ and $p+vpv^2+v^2pv=1$. 
 \pushcounter
 \end{enumerate}
Denote $\gamma=e^{2\pi i/6}$. It is straightforward to check that such $u$ and $v$, together with 
 \[
 w=p+\gamma vpv^2+\gamma^5 v^2pv, 
 \]
satisfy our conditions. 
It will therefore suffice to prove that our conditions imply
\begin{enumerate}
\popcounter
\item $w=p+\gamma q+ \gamma^5 r$, with 
 projections $p,q,r$ satisfying $p+q+r=1$. 
 \item $upu+p=1$ and $p+vpv^2+v^2pv=1$. 
\pushcounter
 \end{enumerate}
 Since $w^6=1$ and $\|w-1\|=1=|e^{2\pi i/6}-1|$,
$\Sp(w)$ is contained in $\{\gamma, 1, \gamma^5\}$, with at least one of $\gamma$ or  $\gamma^5$ belonging to it.  
By \eqref{L.O2.4}, the  unitaries $w^2$ and $\gamma^2 w^2$ are conjugate and 
 $\Sp(w^2)=\{\gamma^2\lambda: \lambda\in \Sp(w^2)\}$. 
Therefore $\Sp(w)=\{\gamma, 1, \gamma^5\}$, and we can write 
\begin{enumerate}
\popcounter
\item \label{L.O2.n} $w=p+\gamma q +\gamma^5 r$
\end{enumerate}
for projections $p,q,r$ satisfying $p+q+r=1$. 
By applying \eqref{L.O2.4} to \eqref{L.O2.n} we obtain $vpv^2=q$ and $vqv^2=r$. 
 In particular, $ p+vpv^2+v^2pv=1$.

 Since $uw^3u=-w^3$, $u$ and $w^3$ generate  a unital copy of $M_2(\bbC)$ and
$p=(w^3+1)/2$ (the equality follows by  \eqref{L.O2.n}) is a projection such that $p+upu=1$. 

Thus $u, v$ and $p=(w^3+1)/2$  satisfy Choi's conditions and the proof is complete. 
\end{proof} 

\begin{lemma} \label{L.type} There is a 4-type $\bfs(\bar x)$ in the language of metric groups such that 
for a unital \cstar-algebra $A$ and a closed group $G$ satisfying $U_0(A)\subseteq G\subseteq A$ 
and $Z(G)=\bbT$ the following are equivalent. 
\begin{enumerate}
\item $\bfs$ is realized in $G$. 
\item   $A$ has a unital subalgebra isomorphic to $\cO_2$ whose unitary group is included in $G$. 
\pushcounter
\end{enumerate}
\end{lemma} 

\begin{proof} 
Define a type $\bfs(\bar x)$ consisting of the following conditions: 
 \begin{enumerate}
 \popcounter
 \item $x_1^2=x_2^3=x_3^6=1$. 
\item\label{I.central}  $\sup_y \|x_4yx_4^{-1} y^{-1}-1\|=0$
\item \label{I.gamma} $\|x_4-1\|=1$
\item\label{I.x4}  $\|x_3-1\|=1$.
\item $x_1x_3^3x_1=-x_3^3$. 
\item $x_2x_3^2x_2^2=x_4^2 x_3^2$. 
\end{enumerate}
Observe that $x_4$ satisfies condition \eqref{I.central} if and only if $x_4\in Z(G)$. Write $\gamma=e^{2\pi i/6}$, and note that 
$\gamma$ and $\gamma^5$ are the only elements of $Z(A)=\bbT$ at the distance exactly 1 from the 
identity. Therefore     \eqref{I.gamma} implies that $x_4=\gamma \cdot 1$ or $x_4=\gamma^5 \cdot 1$. 

If $x_4 = \gamma \cdot 1$, the remaining conditions are satisfied by $x_1, x_2$ and $x_3$ in $G$ if and only if they 
satisfy the assumptions of Lemma~\ref{L.O2.x}. If $x_4 = \gamma^5 \cdot 1$, then $x_1,x_2$ and $x_3^{-1}$ satisfy those conditions. 
Either way, we see that if $G$ realizes the type then by Lemma~\ref{L.O2.x} there exists a unital copy of $\cO_2$ in $A$. 

Since every unitary in $\cO_2$ is of the form $\exp(ia)$ for a self-adjoint $a$ (\cite{Phi:Exponential}),
its unitary group is connected. Therefore if $A$ and $G$ are as above then $A$ has a unital copy of~$\cO_2$ if and only 
if it has a unital copy of $\cO_2$ whose unitary group is included in $G$. By the above, this is equivalent to $\bfs$ being realized in $G$. 
\end{proof}

\begin{proof}[Proof of Theorem~\ref{T2}] 
As in the proof of Theorem~\ref{T1}, let $A$ denote the CAR algebra~$M_{2^\infty}$ and let $g_j$, for $j\in \bbN$, 
 be an enumeration of a dense subgroup of the unitary group~$U(A)$. 
  Fix a unital *-homomorphism $\pi\colon A\to Q$.  
 Since the unitary group of $A$ is connected, 
 we have $\pi(U(A))\subseteq U_0(Q)$ and $\pi(A)'\cap U_0(Q)=U(\pi(A))'\cap U_0(Q)$.

  As in Theorem~\ref{T1}, the type of $(\pi(g_j): j\in \omega)$ does not depend on the choice of $\pi$. 
  Since $Z(A)=\bbC$, Lemma~\ref{L.center} implies $Z(\pi(A)'\cap U_0(Q))=\bbT$.

 Let 4-type $\bfs^+(\bar x)$  
 consist of $\bfs(\bar x)$ as in Lemma~\ref{L.type} together with all conditions of the form 
 \[
 \|g_j x_k g_j^{-1}x_k^{-1}-1\|=0
 \]
 for $j\in \bbN$ and $1\leq k\leq 3$. Then $\bfs^+$ is realized in $U_0(Q)$ if and only if 
 $\bfs$ is realized in $U(\pi(A))'\cap U_0(Q)=\pi(A)'\cap U_0(Q)$. 
 Since the assumptions of  Lemma~\ref{L.type} are satisfied, $\bfs$ is realized in $U(\pi(A))'\cap U_0(Q)$ if and only 
 if $\cO_2$ unitally embeds into $\pi(A)'\cap Q$. 
 Since there are $\pi_1\colon A\to Q$ and $\pi_2\colon A\to Q$ such that $\pi_1(A)'\cap Q$ has a unital 
 copy of~$\cO_2$ and $\pi_2(A)'\cap Q$ does not, our proof is complete. 
 \end{proof} 
 
 We conclude with a remark on the role of the Continuum Hypothesis in the construction of a possible $K$-theory reversing automorphism of the Calkin algebra. 
 Woodin's~$\utilde\Sigma^2_1$ absoluteness theorem (see \cite{Wo:Beyond}) 
 implies that, under a suitable large cardinal assumption, the following holds. 
 If there exists a forcing extension in 
 which the Calkin algebra has a $K$-theory reversing 
 automorphism, then every forcing extension 
 in which the Continuum Hypothesis holds contains a $K$-theory reversing automorphism 
 of the Calkin algebra. This means that if the existence of a $K$-theory reversing automorphism of the Calkin algebra is consistent with ZFC, then it most likely follows from the Continuum Hypothesis. 

\subsection*{Acknowledgment} The results of this note were proved during first author's visit to the Ben Gurion University in May 2015. He would like to thank the Department of Mathematics, and the second author in particular, for their warm hospitality.  We would also like to thank the referee for suggesting several improvements. 
 
\bibliographystyle{amsplain}
\bibliography{Q-not-ctbly-bib}
\end{document}